\documentclass[12pt]{amsart}

\usepackage{amssymb}
\usepackage[all]{xy}
\usepackage[pdftex]{hyperref}
\usepackage{graphicx}

\newtheorem{thm}{Theorem}

\newtheorem{prop}[thm]{Proposition}
\newtheorem{cor}[thm]{Corollary}
\newtheorem{rem}[thm]{Remark}
\theoremstyle{definition}

\newcommand{\sldos}{SL(2,\mathbb{C})}
\DeclareMathOperator{\tr}{tr}
\DeclareMathOperator{\Id}{Id}

\title{The $SU(2)$-character varieties of torus knots}

\author{Javier Mart\'inez}
\address{Facultad de
Matem\'aticas, Universidad Complutense de Madrid, Plaza de Ciencias
3, 28040 Madrid, Spain}

\email{javiermartinez@mat.ucm.es}

\author[V. Mu\~{n}oz]{Vicente Mu\~{n}oz}
\address{Facultad de
Matem\'aticas, Universidad Complutense de Madrid, Plaza de Ciencias
3, 28040 Madrid, Spain}

\email{vicente.munoz@mat.ucm.es}



\thanks{Partially supported through Spanish MICINN grant MTM2010-17389.}

\begin{document}

\maketitle

\begin{abstract}
Let $G$ be the fundamental group of the complement of the torus knot of type $(m,n)$. We study the relationship between $SU(2)$ and $\sldos$-representations of this group, looking at their characters. Using the description of the character variety of $G$, $X(G)$, we give a geometric description of $Y(G)\subset X(G)$, the set of characters arising from $SU(2)$-representations.
\end{abstract}

\section{Preliminaries and notation}

Given a finitely presented group $G=\langle x_1 \ldots x_k \vert r_1,...,r_s \rangle $, a $SU(2)$-representation is a homomorphism $\rho : G \rightarrow SU(2)$.
Every representation is completely determined by the image of the generators, the $k$-tuple $(A_1,...,A_k)$ satisfying the relations $r_j(A_1,...,A_k)=\Id$. It can be shown that the space of all representations, $R_{SU(2)}(G) = \text{Hom}(G,SU(2))$ is an affine algebraic set.

It is natural to declare a certain equivalence relation between these representations: we say that $\rho $ and $\rho ^{\prime}$ are equivalent if there exists $P \in SU(2)$ such that $\rho ^{\prime} (g)=P^{-1} \rho(g) P$ for all $g\in G$.

We want to consider the moduli space of $SU(2)$-representations, the GIT quotient:
$$
\mathcal{M}_{SU(2)} = \text{Hom}(G,SU(2)) // SU(2).
$$
There are also analogous definitions for $\sldos$: we can consider $\sldos$-represen\-tations of $G$, which form a set $R_{\sldos}(G)$, consider $\sldos$-equivalence and construct the associated moduli space:
$$
\mathcal{M}_{\sldos} = \text{Hom}(G,\sldos) // \sldos.
$$
The natural inclusion $SU(2) \hookrightarrow \sldos$ shows that we can regard every $SU(2)$- representation as a $\sldos$-representation.
Moreover, if two representations are $SU(2)$-equivalent, then they are also $SL(2,\mathbb{C})$-equivalent. This leads to a map between moduli spaces:
$$
\mathcal{M}_{SU(2)} \overset{i_{\ast}}{\longrightarrow} \mathcal{M}_{SL(2,\mathbb{C})}
$$
To every representation $\rho \in R_{\sldos}(G)$ we can associate its character $\chi_{\rho}$, defined as the map $\chi_{\rho} : G \rightarrow \mathbb{C}$, $\chi_{\rho}(g) = \tr(\rho(g))$.
This defines a map $\chi : R_{\sldos}(G) \rightarrow \mathbb{C}^{G}$, where equivalent representations have the same character. Its image $X_{\sldos}(G)=\chi(R_{\sldos}(G))$ is called the character variety of $G$.

There is an important relation between the $\sldos$-character variety of $G$ and the moduli space $\mathcal{M}_{\sldos}$. It is seen in \cite{cusha} that:
\begin{itemize}
 \item $X_{\sldos}(G)$ can be endowed with the structure of algebraic variety.
 \item The natural associated map that takes every representation to its character, $\mathcal{M}_{\sldos}(G) \longrightarrow X_{\sldos}(G)$, is bijective.
 We specify the nature of this correspondence for the case of $SU(2)$-representations in the next section.
\end{itemize}
We emphasize that $X_{\sldos}(G)$, as a set, consists of characters of $\sldos$-represen\-tations. We can also take the set of characters of $SU(2)$-representations, and again we will have a map $X_{SU(2)}(G) \overset{i^{\ast}}{\longrightarrow} X_{\sldos}(G)$.

We focus on the case when $G$ is a torus knot group. Consider the torus of revolution $T^{2}\subset S^{3}$. If we identify it with $\mathbb{R}^{2}/ \mathbb{Z}^{2}$, the image of the line
$y=\frac{m}{n}x$ defines the torus knot of type $(m,n)$, $K_{m,n} \subset S^{3}$ for coprime $m,n$. An important invariant of a knot is the fundamental group of its complement in $S^{3}$,
$G_{m,n}=\pi_1(S^{3}-K_{m,n})$. These groups admit the following presentation:
$$
G_{m,n}=\langle x,y \mid x^{m}=y^{n} \rangle
$$

The $SL(2,\mathbb{C})$-character variety of these groups for the case $(m,2)$ was treated in \cite{oll}. A complete description for $(m,n)$ coprime was given in \cite{vic}, and the general case $(m,n)$ was studied using combinatorial tools in \cite{mmoll}. $SU(2)$-character varieties for knot groups were studied in \cite{kla}. For the case $(m,2)$, the relation between both character varieties has been recently treated in \cite{oll2}.

\section{$SU(2)$-character varieties}
We recall that $SU(2) \cong S^{3}$, the isomorphism being given by:
\begin{eqnarray*}
 S^{3} \subset \mathbb{C}^{2}  & \longrightarrow & SU(2) \\ (a,b) & \longrightarrow & \begin{pmatrix}
                                                                                      a & -\bar{b} \\ b & \bar{a}
                                                                                     \end{pmatrix}
\end{eqnarray*}
The correspondence is a ring homomorphism if we look at $S^3$ as the set of unit quaternions.
First of all, we want to point out the following fact, which was already true for $\sldos$:
\begin{prop}
The correspondence:
\begin{eqnarray*}
\mathcal{M}_{SU(2)}(G) & \longrightarrow & X_{SU(2)}(G) \\ \rho & \longrightarrow & \chi_{\rho}
\end{eqnarray*}
that takes a representation to its character is bijective.
\end{prop}
\begin{proof}
We follow the steps taken in \cite{cusha}, this time for $SU(2)$. First of all, every matrix $A$ in $SU(2)$ is normal, hence diagonalizable. Since $\det(A)=1$, the eigenvalues of $A$ are $\{ \lambda, \lambda^{-1} \}$ for some $\lambda\in \mathbb{C}^{\ast}$. In particular, $\tr(A)$ completely determines the set of eigenvalues $\{ \lambda, \lambda^{-1} \}$.

Now, if $\rho$ is a reducible $SU(2)$-representation, there is a common eigenvector $e_1$ for all $\rho(g)$ and therefore they are all diagonal with respect to the same basis. If $\rho'$ is a second reducible representation such that $\chi_{\rho}(g)=\chi_{\rho'}(g)$ for all $g\in G$, this means that they share the same eigenvalues for every $g\in G$. After choosing another basis for $\rho'$ such that $\rho'(g)$ is diagonal for all $g\in G$:
$$
\rho(g)= \begin{pmatrix} \lambda(g) & 0 \\ 0 & \lambda^{-1}(g) \end{pmatrix} \quad \rho'(g)= \begin{pmatrix} \mu(g) & 0 \\ 0 & \mu^{-1}(g) \end{pmatrix}
$$
where either $\lambda(g)=\mu(g)$ or $\lambda(g)=\mu^{-1}(g)$ for every $g\in G$. Interchanging the roles of $\lambda$ and $\lambda^{-1}$ if necessary, there is always $g_1\in G$ such that $\lambda(g_1)=\mu(g_1)$, so there is $g_1\in G$ such that $\rho(g_1)=\rho'(g_1)$. We also notice that if $\rho(g)=\pm \Id$, then $\rho'(g)=\rho(g)= \pm \Id$.

We claim that $\rho(g_2)=\rho'(g_2)$ for all $g_2\in G$. If not, there exists $g_2\in G$ such that $\rho(g_2)=\rho'(g_2)^{-1}\neq \pm \Id$. So $\lambda(g_1)=\mu(g_1)$ and $\lambda(g_2)=\mu^{-1}(g_2)$. On the other hand, we know that $\tr(\rho'(g_1g_2))=\tr(\rho(g_1g_2))$, so:
\begin{eqnarray*}
\mu(g_1)\mu(g_2)+\mu^{-1}(g_1)\mu^{-1}(g_2) & = & \lambda(g_1)\lambda(g_2)+\lambda^{-1}(g_1)\lambda^{-1}(g_2) \\ & = & \mu(g_1)\mu^{-1}(g_2) + \mu^{-1}(g_1) \mu(g_2)
\end{eqnarray*}
Rearranging the terms:
$$
\mu(g_2)(\mu(g_1)-\mu^{-1}(g_1))=\mu^{-1}(g_2)(\mu(g_1)-\mu^{-1}(g_1))
$$
which implies that $\mu(g_2)=\pm 1$, so that $\rho(g_2)=\pm \Id$, a contradiction. Therefore $\lambda(g)=\mu(g)$ for all $g\in G$. Hence there exists $P\in SU(2)$ such that $\rho(g)=P^{-1}\rho(g)P$ for all $g\in G$, i.e, the representations are equivalent.

For the irreducible case, we point out the following fact: if $\rho$ is a irreducible $SU(2)$-representation and $\rho(g) \neq \pm \Id$ for a given $g\in G$, then there exists $h\in G$ such that $\rho$ restricted to the subgroup $H=\langle g,h \rangle$ is again irreducible. To see it, since $\rho(g)\neq \pm \Id$, $\rho(g)$ has two eigenspaces $L_1,L_2$ associated to the pair of different eigenvalues $\mu_1,\mu_2$. Since the representation is irreducible, there are elements $h_i$ such that $L_i$ is not invariant under $\rho(h_i)$. We can take $h=h_1$ or $h=h_2$ unless $L_1$ is invariant under $\rho(h_2)$, or $L_2$ is invariant under $\rho(h_1)$, in this case we can choose $h=h_1h_2$.

For a group generated by two elements, $H= \langle g,h \rangle$, the reducibility of a representation is completely determined by $\chi_{\rho}([g,h])$. It can be seen in the following chain of equivalences:
\begin{eqnarray*}
\rho \vert_{H} \text{ is reducible } & \Leftrightarrow & \rho(g),\rho(h) \text{ share a common eigenvector}\\
& \Leftrightarrow & \rho(g),\rho(h) \text{ are simultaneously diagonalizable } \\
& \Leftrightarrow & [\rho(g),\rho(h)]=\Id \\
& \Leftrightarrow & \tr[\rho(g),\rho(h)]=2 \\
& \Leftrightarrow & \chi_{\rho}([g,h])=2
\end{eqnarray*}

Let $\rho, \rho'$ be two $SU(2)$-representations such  that $\chi_{\rho}=\chi_{\rho'}$. By the previous observation, there are $g,h\in G$ such that $\rho\vert_{\langle g,h \rangle}$ is irreducible, i.e, $\chi_{\rho}([g,h])\neq 2 $. It follows that, since $\chi_{\rho}=\chi_{\rho'}$, $\chi_{\rho'}([g,h])\neq 2$, so $\rho'\vert_{\langle g,h \rangle}$ is irreducible too. Varying $\rho,\rho'$ in their equivalence classes, we can assume that there are basis $B,B'$ such that:
$$
\rho(h)=\rho'(h)=\begin{pmatrix} \lambda & 0 \\ 0 & \lambda^{-1} \end{pmatrix}
$$
The matrices $\rho(g),\rho'(g)$ will not be triangular, by irreducibility, and conjugating again by diagonal unitary matrices, we can assume that:
$$
\rho(g)=\begin{pmatrix} a & -b \\ b & \bar{a} \end{pmatrix}, \qquad \rho'(g)= \begin{pmatrix} a' & -b' \\ b' & \bar{a}'  \end{pmatrix}
$$
for $a,a' \in \mathbb{C},\; b,b'\in \mathbb{R}^{+}$. Notice that $b,b'\neq 0$ since $\rho\vert_{\langle g,h \rangle}$ is irreducible.
More in general, for any $\alpha \in G$:
$$
\rho(\alpha) = \begin{pmatrix} x & -\bar{y} \\y & \bar{x} \end{pmatrix}, \qquad \rho'(\alpha)= \begin{pmatrix} x' & -\bar{y}' \\ y' & \bar{x}' \end{pmatrix}
$$
Now, the equations $\chi_{\rho}(\alpha) = \chi_{\rho'}(\alpha), \: \chi_{\rho}(h\alpha)=\chi_{\rho'}(h\alpha)$ imply that:
\begin{eqnarray*}
x+\bar{x} & = & x'+\bar{x}' \\ \lambda x + \lambda^{-1} \bar{x} & = &  \lambda x' + \lambda^{-1} \bar{x}'
\end{eqnarray*}
and since $\lambda \neq \pm 1$, we get that $x=x'$.

Substituting $\alpha=g$, we get that $a=a'$ and since $\det(\rho(g))=\det(\rho'(g))=1$, $b=b'$, so $\rho(g)=\rho'(g)$.

Substituting again $\alpha$ for $g\alpha$, we arrive at the equation $ax-by=ax-by'$, which implies that $y=y'$ and finally that $\rho(\alpha)=\rho'(\alpha)$: we have proved that the representations $\rho$ and $\rho'$, after $SU(2)$-conjugation, are the same, i.e, they are equivalent.
\end{proof}
\begin{cor}
We have a commutative diagram:
$$
\xymatrix{\mathcal{M}_{SU(2)}(G) \ar[r]^{1:1} \ar[d]^{i_{\ast}} & X_{SU(2)}(G)\ar[d]^{i_{\ast}} \\ \mathcal{M}_{\sldos}(G) \ar[r]^{1:1} & X_{\sldos}(G)}
$$
\end{cor}
The previous corollary shows that we can equivalently study the relationship between $SU(2)$ and $\sldos$-representations of $G$ from the point of view of their characters or from the point of view of their representations. Looking at the diagram, we also deduce that:
\begin{cor}
\label{inclusionmod}
The natural inclusion $i_{\ast}: \mathcal{M}_{SU(2)}(G) \longrightarrow \mathcal{M}_{SL(2,\mathbb{C})}(G)$ is injective.
\end{cor}

\section{$SU(2)$-character varieties of torus knots}

We focus now on the specific case of the torus knot $G_{m,n}$ of coprime type $(m,n)$. Henceforth, we will often denote $X_{\sldos}=X_{\sldos}(G)$ and omit the group in our notation.
In this case:
$$
R_{\sldos}(G)= \{ (A,B) \in SL(2,\mathbb{C}) \mid A^m=B^n \}
$$
and:
$$
R_{SU(2)}(G) = \{ (A,B) \in SU(2) \mid A^m=B^n \}
$$
We have a decomposition of $X_{\sldos}$:
$$
X_{\sldos} = X_{red} \cup X_{irr}
$$
where $X_{red}$ is the subset of characters of reducible representations and $X_{irr}$ is the subset of  characters of irreducible representations. Inside $X_{\sldos}$ we have $i_{\ast}(X_{SU(2)})$, i.e, the set of characters of $SU(2)$-representations. For simplicity, we will denote $Y= i_{\ast}(X_{SU(2)})$. Again, $Y$ decomposes in $Y_{red} \cup Y_{irr}$.

\subsection*{Reducible representations}

\begin{prop}
\label{redulemma}
There is an isomorphism $Y_{red}\cong [-2,2] \subset \mathbb{R}$
\end{prop}
\begin{proof}
 We will use, from now on, the explicit description of $X_{\sldos}$ given in \cite{vic}.
There is an isomorphism $X_{red} \cong \mathbb{C}$ given by:
$$
\left( A= \begin{pmatrix}
           t^n & 0 \\ 0 & t^{-n}
          \end{pmatrix},
B = \begin{pmatrix}
     t^m & 0 \\ 0 & t^{-m}
    \end{pmatrix}
\longrightarrow s=t+t^{-1} \in \mathbb{C}
\right)
$$
This is because given a reducible $\sldos$-representation $\rho$, we can consider the associated split representation $\rho=\rho' + \rho''$, for which in
a certain basis takes the form:
$$
A = \begin{pmatrix} \lambda & 0 \\ 0 & \lambda^{-1} \end{pmatrix}, B= \begin{pmatrix} \mu & 0 \\ 0 & \mu^{-1} \end{pmatrix}
$$
and the equality $A^{m}=B^{n}$ implies that $\lambda=t^{n}, \mu=t^{m}$ for a unique $t\in \mathbb{C}$ (here we use that $m,n$ are coprime).
Now, since $A,B \in SU(2)$, $t$	 must satisfy that $\vert t \vert^{2}=1$, i.e, $t \in S^{1} \subset \mathbb{C}$. We have to also take account of the
change of order of the basis elements and therefore $t \sim \frac{1}{t}$. So the parameter space is isomorphic to $[-2,2]$ (under the correspondence $t \in S^{1} \longrightarrow s=t+t^{-1}=2\operatorname{Re}(t)\in [-2,2]$).
\end{proof}
To explicitly describe when a pair $(A,B)$ is reducible, we follow \cite[2.2]{vic}.
First of all, $A$ and $B$ are diagonalizable (recall that $A,B\in SU(2)$), so we can rule out the Jordan type case since it is not possible. So:
\begin{prop}
\label{reducases}
In any of the cases:
\begin{itemize}
\item $A^{m}=B^{n}\neq \pm \Id$
\item $A=\pm \Id$ or $B=\pm \Id$
\end{itemize}
the pair $(A,B)$ is reducible.
\end{prop}
\begin{proof}
Let us deal with the first case, when $A^{m}=B^{n} \neq \pm \Id $. $A$ is diagonalizable with respect to a basis $\{e_{1},e_{2}\}$, and takes the form $\begin{pmatrix} \lambda & 0 \\ 0 & \lambda^{-1} \end{pmatrix}$.
Then:
$$
B^{n}=A^{m}= \begin{pmatrix} \lambda^{m} & 0 \\ 0 &\lambda^{-m} \end{pmatrix}
$$
so $B$ is diagonal in the same basis and the pair is reducible.
For the second case, if $A= \alpha \Id$, where $\alpha=\pm 1$, then any basis diagonalizing $B$ diagonalizes $A$, hence the pair is reducible. The case $B=\alpha \Id$ follows in the same way.
\end{proof}

\subsection*{Irreducible representations}

\hspace*{10pt}

Now we look at the irreducible set of representations, since we want to study $Y_{irr}$.
Let $(A,B)\in R_{SU(2)}(G)$ be an irreducible pair. Both are diagonalizable, and using Proposition \ref{reducases}, they must satisfy that $A^{m}=B^{n}=\pm\Id$, $A,B \neq \pm \Id$. The eigenvalues $\lambda, \lambda^{-1} \neq \pm 1$ of $A$ satisfy $\lambda^{m} = \pm 1$, the eigenvalues $\mu, \mu^{-1}$ of $B$ satisfy $\mu^{n}= \pm 1$ and $\lambda^{m}= \mu^{n}$.

We can associate to $A$ a basis $\{ e_{1},e_{2} \}$ under which it diagonalizes, and the same for $B$, obtaining another basis $ \{ f_{1}, f_{2} \}$. The eigenvalues $\lambda,\mu$ and the eigenvectors ${e_i,f_i}$ completely determine the representation $(A,B)$.
We are interested in $i_{\ast}(\mathcal{M}_{SU(2)})$, $SL(2,\mathbb{C})$-equivalence classes of such pairs $(A,B)$, and these are fully described by the projective invariant of the four points $\{ e_1, e_2,f_1,f_2 \}$, the cross ratio:
$$
[e_1,e_2,f_1,f_2] \in \mathbb{P}^{1}- \{0,1,\infty \}
$$
(we may assume that the four eigenvectors are different since the representation is irreducible, see \cite{vic} for details).

Since both $A,B \in SU(2)$, we know that $e_1 \perp e_2 $ and $\Vert e_1 \Vert = \Vert e_2 \Vert = 1$, so shifting the vectors by a suitable rotation $C\in SU(2)$, we can assume that $e_1 = [1:0], e_2 = [0:1]$, and therefore $f_1 = [a:b], f_2 = [-\bar{b} : \bar{a}]$, since they are orthogonal too.
So the pair $(A,B)$ inside $X_{\sldos}$ is determined by $\lambda$, $\mu$ satisfying the conditions above and the projective cross ratio:
$$
r = \Big[ e_1,e_2,f_1,f_2 \Big] = \Big[ 0,\infty, \frac{b}{a}, - \frac{\bar{a}}{\bar{b}}\Big] = \frac{b\bar{b}}{-a\bar{a}} = \frac{b\bar{b}}{b\bar{b}-1} = \frac{t}{t-1}
$$
where we have used that $a\bar{a}+b\bar{b}=1$ and $t=\vert b \vert^{2}, b\in (0,1)$. We also get that $r$ is real and $r \in (-\infty, 0)$.

The converse is also true: if the triple $(\lambda, \mu , r)$, satisfies that $\lambda^{m}= \mu^{n}= \pm 1$, $\lambda,\mu\neq \pm 1$ and $r\in (-\infty, 0)$, then $(A,B) \in i_{\ast}(\mathcal{M}_{SU(2)})$.
To see this, $r$ determines uniquely $t= \vert b \vert^{2}$ since $r(t)$ is invertible for $t \in (0,1)$. Once $ \vert b \vert$ is fixed, we get that $\vert a \vert$ is fixed too, using $\vert a \vert^{2} = 1 - \vert b \vert^{2}$. We can choose any $(a,b) \in S^{1} \times S^{1}$ and we conclude that $(A,B)$ is $SL(2,\mathbb{C})$-equivalent to a $SU(2)$ representation. To be  more precise, it is equivalent to the representation with eigenvalues $\lambda,\mu$ and eigenvectors $[1:0],[0:1],[a:b],[-\bar{b}, \bar{a}]$.

Finally, we have to take account of the $ \mathbb{Z}_{2}\times \mathbb{Z}_{2} $ action given by the permutation of the eigenvalues:
\begin{itemize}
\item Permuting $e_1,e_2$ takes $(\lambda,\mu,r)$ to $(\lambda^{-1},\mu,r^{-1})$
\item Permuting $f_1,f_2$ takes $(\lambda,\mu,r)$ to $(\lambda,\mu^{-1},r^{-1})$
\end{itemize}

Since $\lambda^{m}=\mu^{n}=\pm 1$, we get that:
\begin{equation}\label{eqn:labeleada}
\lambda=e^{\pi ik/m}, \qquad \mu=e^{\pi ik'/n},
\end{equation}
where since $\lambda \sim \lambda^{-1}, \mu \sim \mu^{-1}$ and $\lambda \neq \pm 1, \mu \neq \pm 1$, we can restrict to the case when $0<k<m, 0<k'<n$. We also notice that $\lambda^{m}=\mu^{n}$ implies that $k\equiv k' \pmod{2}$. So the irreducible part is made of $(m-1)(n-1)/2$ intervals.

We have just proved:
\begin{prop}
\label{irrchar}
$$
Y_{irr} \cong \lbrace (\lambda,\mu,r) : \lambda^{m}=\mu^{n}=\pm 1; \lambda,\mu\neq \pm 1; r\in (-\infty,0)\rbrace / \mathbb{Z}_{2}\times \mathbb{Z}_{2}
$$
This real algebraic variety consists of $\frac{(m-1)(n-1)}{2}$ open intervals.
\end{prop}

To describe the closure of the irreducible orbits, we have to consider the case when $e_1=f_1$, since this is what happens in the limit (the situation is analogous when $e_2=f_2$). In this situation $r=0$, and the representation is equivalent to a reducible representation. Taking into account Lemma \ref{redulemma}, it corresponds to a certain $t\in S^{1}$  such that $\lambda=t^{n}$, $\mu=t^{m}$.
We have another limit case $r=-\infty$, if we allow $e_1=f_2$. The representation is again reducible and corresponds to another $t'\in S^{1}$ such that $\lambda=(t')^{n}$, $\mu^{-1}=(t')^{m}$.

\begin{rem}
The explicit description of the set of $SU(2)$-representations allows us to give an alternative proof of Corollary \ref{inclusionmod}, which stated that the inclusion $i_{\ast}:\mathcal{M}_{SU(2)}\rightarrow \mathcal{M}_{SL(2,\mathbb{C})}$ is injective.

Let us see this. Suppose that $(A,B)$ and $(A',B')$ are two $SU(2)$-representations which are mapped to the same point in $\mathcal{M}_{SL(2,\mathbb{C})}$, i.e, which are $SL(2,\mathbb{C})$-equivalent.
If we denote by ${u_1,u_2,u_3,u_4}$ the set of eigenvectors of $(A,B)$ and by ${v_1,v_2,v_3,v_4}$ the set of eigenvectors of $(A',B')$, we know that:
$$
[u_1,u_2,u_3,u_4]=[v_1,v_2,v_3,v_4]=r\in (-\infty,0)
$$
Since their cross ratio is the same, we know that there exists $P\in SL(2,\mathbb{C})$ that takes the set ${u_i}$ to ${v_i}$. Moreover, since $P$ takes the unitary basis ${u_1,u_2}$ to the unitary basis ${v_1,v_2}$, we get that $P\in SU(2)$, and therefore both representations are $SU(2)$-equivalent.
\end{rem}

\subsection*{Topological description}
We finally describe $Y$ topologically. We refer to \cite{vic} for a geometric description of $X_{\sldos}$.

Using proposition \ref{irrchar}, $Y_{irr}$ is a collection of real intervals (parametrized by $r\in (-\infty,0)$) for a finite number of $(\lambda,\mu)$ that satisfy the required conditions.
By our last observation, the limit cases when $r=0,\infty$ (i.e, points in the closure of $Y_{irr}$)  correspond to the points where $Y_{irr}$ intersects $Y_{red}$.

As we saw before, each interval has two points in its closure: these are $t_0\in S^{1}$ such that $t_0^{n}=\lambda$, $t_0^{m}=\mu$ $(r=0)$ and $t_1\in S^{1}$ corresponding to $t_1^{n}=\lambda$, $t_1^{m}=\mu^{-1}$ $(r=-\infty)$. The conditions on $\lambda,\mu$ force that $t_0\neq t_1$ so that we get different intersection points with $Y_{red}$.

$Y$ is topologically a closed interval ($Y_{\text{red}})$ with $(m-1)(n-1)/2$ open intervals ($Y_{\text{irr}}$) attached at $(m-1)(n-1)$ different points (without any intersections among them). The interval $Y_{red}=[-2,2]$ sits inside $X_{red} \cong \mathbb{C}$ and every real interval in $Y_{irr}$ is inside the corresponding complex line in $X_{irr}$.

The situation is described in the following two pictures:
\begin{center}
 \begin{figure}[h]
 \includegraphics[width=5 cm]{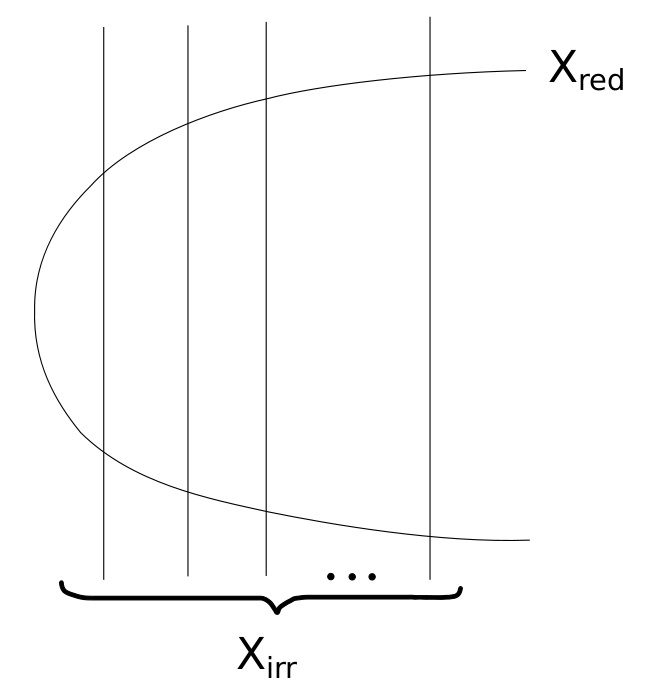}
 \caption{Picture of $X_{\sldos}$, defined over $\mathbb{C}$. The drawn lines are curves isomorphic to $\mathbb{C}$.}
 \end{figure}
\end{center}

\begin{center}
 \begin{figure}[h]
 \includegraphics[width=5 cm]{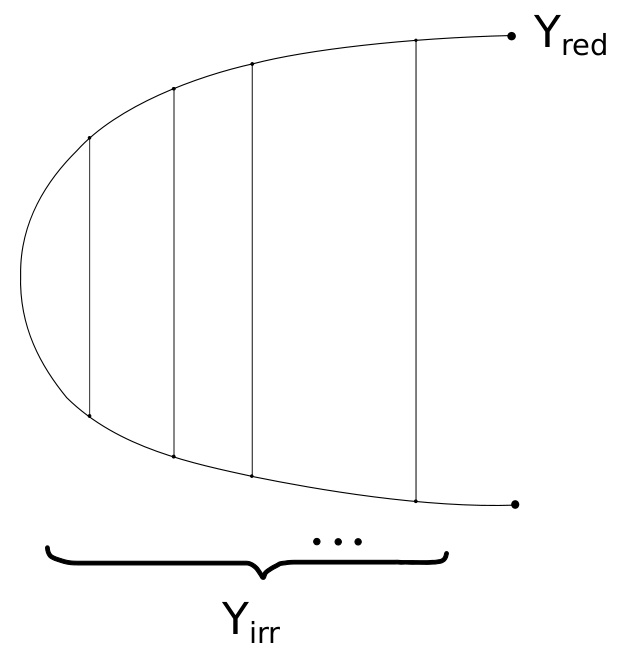}
 \caption{Picture of $Y\subset X_{\sldos}$, defined over $\mathbb{R}$. The picture displays the set of real segments which form $Y_{irr}$.}
 \end{figure}
\end{center}

\section{Noncoprime case}
If $\text{gcd}(m,n)=d>1$, then $G_{m,n}$ does no longer represent a torus knot, since these are only defined in the coprime case. However, the group $G_{m,n}=\langle x,y \mid x^n=y^m \rangle$ still makes sense and we can study the representations of this group into $\sldos$ and $SU(2)$ using the method described above.
We will denote by $a,b$ the integers that satisfy:
\begin{eqnarray*}
m &=&a\, d, \\ n&=&b\,d.
\end{eqnarray*}
As we did before, we focus on $Y=i_{\ast}(X_{SU(2)})$, the set of characters of $SU(2)$-representations.
\subsection*{Reducible representations}
First of all, we describe what happens in the $\sldos$ case:

\begin{prop}
There is an isomorphism:
 $$
 X_{red}\cong \bigsqcup_{i=0}^{\lfloor d/2 \rfloor}X_{red}^{i}
 $$
where:
 \begin{itemize}
 \item[-] $X_{red}^{i}\cong \mathbb{C}^{\ast}$ for $0<i<\frac{d}{2}$.
 \item[-] $X_{red}^{i}\cong \mathbb{C}$ for $i=0$ and $i=\frac{d}{2}$ if $d$ is even
 \end{itemize}
\end{prop}

\begin{proof}
As it is shown in \cite{vic}, an element in $X_{red}$ can be regarded as the character of a split representation, $\rho=\rho'\oplus \rho'^{-1}$. There is a basis such that:
 $$
 A= \begin{pmatrix}
           \lambda & 0 \\ 0 & \lambda^{-1}
          \end{pmatrix}, \quad
 B = \begin{pmatrix}
     \mu & 0 \\ 0 & \mu^{-1}
    \end{pmatrix},
 $$
where $A^{m}=B^{n}$ implies that $\lambda^{m}=\mu^{n}$. We deduce that $(\lambda^{a})^{d}=(\mu^{b})^{d}$, so that $(\lambda,\mu)$ belong to one of the components:
 $$
 X_{red}^{i}=\lbrace (\lambda,\mu) \vert \lambda^{a}=\xi^{i}\mu^{b} \rbrace = \lbrace (\lambda,\mu) \vert
 \lambda^{a}\mu^{-b}=\xi^{i} \rbrace,
 $$
where $\xi$ is a primitive $d$-th root of unity. These components are disjoint, and each one of them is parametrized by $\mathbb{C}^{\ast}$. To see this, let us fix a component, $X_{red}^{i}$, and let $\alpha$ be a $b$-th root of $\xi^{i}$.
Then:
 \begin{eqnarray*}
 X_{red}^{i} & = & \lbrace (\lambda,\mu) \vert \lambda^{a}=\xi^{i}\mu^{b} \rbrace \\
 & = & \lbrace (\lambda,\mu) \vert \lambda^{a}=\alpha^{b}\mu^{b} \rbrace \\
 & = & \lbrace (\lambda,\nu) \vert \lambda^{a}=\nu^{b} \rbrace \cong \mathbb{C}^{\ast}\, .
 \end{eqnarray*}
In other words, for each $(\lambda,\mu) \in X_{red}^{i}$ there is a unique $t\in \mathbb{C}^{\ast}$ such that
$t^{b}=\lambda$, $t^{a}=\alpha\mu$.
However, we have to take account of the action given by permuting the two vectors in the basis, which corresponds to
the change $(\lambda,\mu)\sim (\lambda^{-1},\mu^{-1})$.
In our decomposition, if $(\lambda,\mu)\in X_{red}^{i}$, then $(\lambda^{-1},\mu^{-1}) \in X_{red}^{-i}$.
So $t \in X_{red}^{i}$ is equivalent to $1/t \in X^{-i}_{red}$.

For $0\leq i \leq d-1$, we have two possibilities.
If $i\not\equiv-i \pmod{d}$, then $X^{i}_{red}$ and $X^{-i}_{red}$ get identified.
If $i\equiv -i \pmod{d}$, then $t \sim t^{-1} \in X_{red}^{i}\cong \mathbb{C}$,  and thus
$X_{red}^{i}/_{\sim} \cong \mathbb{C}^{\ast}/_{a\sim a^{-1}} \cong \mathbb{C}$.

When $d$ is even, there are two $i\in \mathbb{Z}/d\mathbb{Z}$ such that $i\equiv -i \pmod{d}$,
so we get two copies of $\mathbb{C}$ in $Y_{red}$. When $d$ is odd we get just one,
since there is only one solution ($i\equiv 0$). The remaining copies of $X_{red}^{i}$ get identified pairwise: $X_{red}^{i}\sim X_{red}^{-i}$.
\end{proof}

Now, for the case of $SU(2)$-representations, we have:
\begin{prop}
There is an isomorphism:
 $$Y_{red}\cong \bigsqcup_{i=0}^{\lfloor \frac{d}{2} \rfloor} Y_{red}^{i} $$
where:
\begin{itemize}
\item[-] $Y_{red}^{i}\cong S^{1}$ for $0<i<\frac{d}{2}$
\item[-] $Y_{red}^{i}\cong [-2,2]$ for $i=0$, $i=\frac{d}{2}$ if $d$ is even
\end{itemize}
\end{prop}

\begin{proof}
If $(A,B)$ is a reducible $SU(2)$-representation, both are diagonalizable with respect to a certain basis and therefore:
$$
A= \begin{pmatrix}
           \lambda & 0 \\ 0 & \lambda^{-1}
          \end{pmatrix} \quad
B = \begin{pmatrix}
     \mu & 0 \\ 0 & \mu^{-1}
    \end{pmatrix}
$$
The equality $A^{m}=B^{n}$ gives us that $\lambda^{m}=\mu^{n}$. So the pair $(\lambda,\mu)$ belongs to a certain component $X_{red}^{i}$. Since it is a $SU(2)$-representation, the eigenvalues $\lambda$ and $\mu$ satisfy that $\vert \lambda \vert = \vert \mu \vert = 1$. This implies that $(\lambda,\mu)\in S^{1}\subset \mathbb{C}^{\ast} \cong X_{red}^{i}$: we define $Y_{red}^{i}:=S^{1}\subset X_{red}^{i}$.

We have to take into account the equivalence relation in $X_{red}$ given by the permutation of the eigenvectors. If $i\not\equiv -i \pmod{d}$, then $Y_{red}^{i}\cong Y_{red}^{-i}$. If $i \equiv -i \pmod{d}$, then $Y_{red}^{i}\cong S^{1}/_{a\sim a^{-1}} \cong [-2,2]$. This gives the desired result.
\end{proof}

\subsection*{Irreducible representations}
We start by describing what happens in the $SU(2)$ case.

\begin{prop}
We have an isomorphism
 $$
 Y_{irr}\cong \{ (\lambda,\mu,r) : \lambda^{m}=\mu^{n}=\pm 1; \lambda,\mu\neq \pm 1, r\in
 (-\infty,0)\}/\mathbb{Z}_2\times\mathbb{Z}_2\, .
 $$
This real algebraic variety consists of:
 \begin{itemize}
 \item $\frac{(m-1)(n-1)+1}{2}$ open intervals if $m,n$ are both even,
 \item $\frac{(m-1)(n-1)}{2}$ open intervals in any other case.
 \end{itemize}
\end{prop}

\begin{proof}
By Proposition \ref{reducases}, a representation $(A,B)$ is reducible unless $A^{m}=B^{n}=\pm \Id$,
$A,B \neq \pm \Id$. So the set of irreducible representations can be described using the same tools
as before: the set of equivalence classes of irreducible representations is a collection of intervals
$r\in (-\infty, 0)$ parametrized by pairs $(k,k')$ satisfying:
 \begin{equation} \label{conditionkas}
 0 <k<m, \quad 0<k'<n, \quad k\equiv k' \pmod{2}.
 \end{equation}
We compute the number of such pairs, separating in three different cases according to the parity of $m$ and $n$:

Suppose $m,n$ are both even. If $k\equiv k' \equiv 0 \pmod{2}$, then $k\in \lbrace 2, 4,\ldots, m-2 \rbrace$,
$k' \in \lbrace 2, 4, \ldots n-2 \rbrace$, so there are $\frac{(m-2)(n-2)}{4}$ such pairs. If
$k\equiv k' \equiv 1 \pmod{2}$, $k \in \lbrace 1,3, \ldots m-1 \rbrace$, $k' \in \lbrace 1,3, \ldots, n-1\rbrace$,
we have $\frac{mn}{4}$ pairs. The sum is $\frac{(m-2)(n-2)}{4}+ \frac{mn}{4}=\frac{(m-1)(n-1)+1}{4}$.

Suppose $m$ is even and $n$ is odd (the case $m$ odd and $n$ even is similar). Then if
$k\equiv k' \equiv 0 \pmod{2}$, $k\in \lbrace 2, 4, \ldots, m-2 \rbrace$, $k' \in \lbrace 2, 4, \ldots n-1 \rbrace$,
we get $\frac{(m-2)(n-1)}{4}$ such pairs. If $k\equiv k' \equiv 1 \pmod{2}$, $k \in \lbrace 1,3, \ldots m-1 \rbrace$, $k' \in \lbrace 1,3, \ldots, n-2 \rbrace$, and there are $\frac{m(n-1)}{4}$ such pairs.
We get in total $\frac{m(n-1)}{4}+\frac{(m-2)(n-1)}{4}= \frac{(m-1)(n-1)}{2}$.

Finally, suppose both $m,n$ odd. If $k\equiv k' \equiv 0 \pmod{2}$, $k\in \lbrace 2, 4, \ldots, m-1 \rbrace$,
$k' \in \lbrace 2, 4, \ldots n-1 \rbrace$, and we get $\frac{(m-1)(n-1)}{4}$ such pairs. If $k\equiv k'
\equiv 1 \pmod{2}$, $k \in \lbrace 1,3, \ldots m-2 \rbrace$, $k' \in \lbrace 1,3, \ldots, n-2 \rbrace$,
there are $\frac{(m-1)(n-1)}{4}$ such pairs. We get $\frac{(m-1)(n-1)}{2}$ pairs in total.

We have obtained a decomposition:
$$
Y_{irr}=\bigsqcup_{k,k'}Y_{irr}^{(k,k')}
$$
where every $Y_{irr}^{(k,k')}$ is an open interval isomorphic to $(-\infty,0)$.
\end{proof}

For the case of $SL(2,\mathbb{C})$ representations, we have the following:

\begin{prop}
The component $X_{irr}\subset X_{SL(2,\mathbb{C})}$ is described as
 $$
 X_{irr} = \bigsqcup_{k,k'} X_{irr}^{(k,k')}
 $$
 where $k,k'$ satisfy \eqref{conditionkas}, and $X_{irr}^{(k,k')} =\mathbb{P}^1-\{0,1,\infty\}$.
This complex algebraic variety consists of $\frac{(m-1)(n-1)+1}{2}$ components if $m,n$ are both even,
of $\frac{(m-1)(n-1)}{2}$ components if one of $m,n$ is odd.
Moreover $Y_{irr}^{(k,k')}=(-\infty,0)\subset X_{irr}^{(k,k')}$ in the natural way.
\end{prop}

The limit cases $r=0$, $r=-\infty$ correspond to the closure of the irreducible components, and these points
are exactly where $\overline{Y}_{irr}$ intersects $Y_{red}$. The triples $(\lambda,\mu,0)$, $(\lambda,\mu,-\infty)$
correspond to the reducible representations with eigenvalues $(\lambda,\mu)$ and $(\lambda,\mu^{-1})$. Since $\lambda,\mu\neq \pm 1$, we get two different intersection points. Note that the pattern of intersections for
$\overline{X}_{irr}$ and $X_{red}$ is the same, but the components are complex algebraic varieties now.

To understand the way the closure of the components of $Y_{irr}$ intersect $Y_{red}$, we have the following:

\begin{prop} \label{prop:11}
The closure of $Y_{irr}^{(k,k')}$ is a closed interval that joins $Y_{red}^{i_{0}}$ with $Y_{red}^{i_{i}}$, where:
 $$
 i_{0}=\frac{k-k'}{2}, \quad i_{i}=\frac{k+k'}{2} \pmod{d}\, .
 $$
\end{prop}

\begin{proof}
Set $D=2d\ ab$, and consider $\omega$ a primitive $D$-th root of unity. Then
$\xi:=\omega^{D/d}=\omega^{2ab}$ is a primitive $d$-th root of unity. The irreducible component
$Y_{irr}^{(k,k')}$ is the interval $(\lambda,\mu,r)$, $r\in (-\infty,0)$, where
 $$
 \lambda = (\omega^{b})^{k}, \quad \mu=(\omega^{a})^{k'},
 $$
and $k,k'$ are subject to the conditions (\ref{conditionkas}), see equation (\ref{eqn:labeleada}).
The points in the closure of $Y_{irr}^{(k,k')}$ correspond to the reducible representations with eigenvalues $(\lambda,\mu)$ and $(\lambda,\mu^{-1})$. Clearly $(\lambda,\mu) \in X_{red}^{i_0}$, since
 $$
 \lambda^{a}\mu^{-b}=\omega^{kab}\omega^{-k'ab}=\omega^{\frac{k-k'}{2}2ab}=\omega^{i_{0}2ab}=\xi^{i_{0}}\, ,
 $$
and $(\lambda,\mu^{-1}) \in X_{red}^{i_{1}}$, since
 $$
 \lambda^{a}\mu^{b}=\omega^{kab}\omega^{k'ab}=\xi^{i_{1}}\, .
 $$
\end{proof}

Proposition \ref{prop:11} gives a clear rule to depict $Y=Y_{irr}\cup Y_{red}$ for every pair $(m,n)$.
Acutally, $Y$ is a collection of intervals attached on their endpoints to $Y_{red}$, which consists of
several disjoint copies of $S^{1}$ and $[-2,2]$. Note that the pattern of intersections for the irreducible
components of $X_{SL(2,\mathbb{C})} =X_{irr}\cup X_{red}$ is the same as that of $Y$.

When $m,n$ are coprime, we recover our previous pictures.

\begin{cor} \label{connectingcomp}
For any two different components $Y_{red}^{i_{0}}, Y_{red}^{i_{1}} \subset Y_{red}$, there is a pair $(k,k')$ such that $\overline{Y}_{irr}^{(k,k')}$ joins them.

In particular, $Y$ is a connected topological space.
\end{cor}

\begin{proof}
We can assume $0\leq i_{0} < i_{i} \leq \frac{d}{2}$. Then $0<k=d+i_{0}-i_{1}<d\leq m$ and $0<k'=d-i_{0}-i_{1}<d
\leq n$ both satisfy that $k\equiv k' \pmod{2}$ and $\frac{k-k'}{2}=i_{0}$, $\frac{k+k'}{2}=i_{1}$.
\end{proof}

\begin{rem}
It can be checked that there is no component $\overline{Y}_{irr}^{(k,k')}$ which joins $Y_{red}^{i_{0}}$ to
itself when $m=n$, or when one of $m,n$ divides the other, and we are dealing with $i_0=0$ or $i_0=d/2$ (the latter
only if $d$ is even).

Actually, such component would correspond to a pair $(k,k')$ such that $\frac{k-k'}{2}\equiv \pm i_0 \pmod{d}$
and $\frac{k+k'}{2}\equiv \pm i_0 \pmod{d}$. Accounting for all possibilities of signs, we have
either $k\equiv \pm 2i_0$, $k'\equiv 0 \pmod{d}$, or $k\equiv 0$, $k'\equiv \pm 2i_0 \pmod{d}$.
This has solutions unless $m>n=d$, $i_0=0,d/2$; $n>m=d$, $i_0=0,d/2$; or $m=n=d$, any $i_0$.
 \end{rem}

\end{document}